\documentclass[11pt,a4paper,leqno]{amsart}

\title[Gabor wave front set of compactly supported distributions]
{The Gabor wave front set of compactly supported distributions}

\author[P. Wahlberg]{Patrik Wahlberg}

\address{Department of Mathematics, Linn{\ae}us University, SE--351 95 V\"axj\"o, Sweden}

\email{patrik.wahlberg[AT]lnu.se}

\usepackage{tikz,xifthen}

\usepackage{latexsym}
\usepackage[a4paper]{geometry}
\usepackage{amsmath}
\usepackage{amssymb}
\usepackage{amsthm}
\usepackage{bbm}
\usepackage{amsfonts}
\usepackage{mathrsfs}
\usepackage{calc}
\usepackage{cite}
\usepackage{color}
\usepackage{graphicx}
\usepackage{enumerate}

\numberwithin{equation}{section}          

\newtheorem{thm}{Theorem}
\numberwithin{thm}{section}

\newcommand{\rubrik}{}
\newtheorem{prop}[thm]{Proposition}
\newtheorem{cor}[thm]{Corollary}

\theoremstyle{definition}

\newtheorem{defn}[thm]{Definition}

\theoremstyle{remark}

\newcommand{\PWcc}[1]{\mathbf C^{#1}}
\newcommand{\PWrr}[1]{\mathbf R^{#1}}
\newcommand{\PWnn}[1]{\mathbf N^{#1}}
\newcommand{\PWro}{\mathbf R}
\newcommand{\PWco}{\mathbf C}
\newcommand{\PWno}{\mathbf N}
\newcommand{\PWzo}{\mathbf Z}
\newcommand{\PWSp}{\operatorname{Sp}}
\newcommand{\PWGL}{\operatorname{GL}}
\newcommand{\PWMp}{\operatorname{Mp}}
\newcommand{\PWWF}{\mathrm{WF}}
\newcommand{\PWcS}{\mathscr{S}}
\newcommand{\PWcF}{\mathscr{F}}
\newcommand{\PWcE}{\mathscr{E}}
\newcommand{\PWJ}{\mathcal{J}}
\newcommand{\PWep}{\varepsilon}
\newcommand{\PWeabs}[1]{\langle #1\rangle}
\newcommand{\PWre}{{\rm Re} \, }
\newcommand{\PWim}{{\rm Im} \, }
\newcommand{\PWKer}{\operatorname{Ker}}

\begin{document}

\begin{abstract}
We show that the Gabor wave front set of a compactly supported distribution equals zero times the projection on the second variable of the classical wave front set.
\end{abstract}


\maketitle

\centerline{\emph{Dedicated to Luigi Rodino on the occasion of his 70$^{th}$ birthday}}

\section{Introduction}
\label{PWsec:introduction}

The Gabor wave front set of tempered distributions was introduced by H\"ormander 1991 \cite{PWHormander1}.
The idea was to measure singularities of tempered distributions both in terms of smoothness and decay at infinity comprehensively. 
Using the short-time Fourier transform the Gabor wave front set can be described as the directions in phase space $T^* \PWrr d$ where a distribution $u \in \PWcS'(\PWrr d)$ does not decay rapidly. 
The Gabor wave front set of $u \in \PWcS'(\PWrr d)$ is empty exactly if $u \in \PWcS(\PWrr d)$. 

The Gabor wave front set behaves differently than the classical $C^\infty$ wave front set, also introduced by H\"ormander 1971 (see \cite[Chapter~8]{PWHormander0}). 
It is for example translation invariant. 
The Gabor wave front set is adapted to the Shubin calculus of pseudodifferential operators \cite{PWShubin1} where symbols have isotropic behavior in phase space. 
With respect to this calculus, the corresponding notion of characteristic set, and the Gabor wave front set, pseudodifferential operators are microlocal and microelliptic, similar to pseudodifferential operators with H\"ormander symbols in their natural context. 

The main result of this note concerns the Gabor wave front set of compactly supported distributions. 
We show (see Corollary \ref{PWcor:WFGcompact}) that for $u \in \PWcE'(\PWrr d)$ we have
\begin{equation}\label{PWeq:mainresult}
\PWWF_G (u) = \{ 0 \} \times \pi_2 \PWWF(u) 
\end{equation} 
where $\PWWF_G$ denotes the Gabor wave front set, $\PWWF$ denotes the classical wave front set, 
and $\pi_2$ is the projection on the covariable (second) phase space $\PWrr d$ coordinate. 
By \cite[Theorem~8.1.3]{PWHormander0}, $\pi_2 \PWWF(u) = \Sigma(u)$.
The symbol $\Sigma(u)$ denotes the cone complement of the 
space of directions in $\PWrr d$ in an open conic neighborhood of which
the Fourier transform of $u \in \PWcE'(\PWrr d)$ decays rapidly. 

The equality \eqref{PWeq:mainresult} thus describes exactly the Gabor wave front set of $u \in \PWcE'(\PWrr d)$ in terms of known ingredients in terms of $\PWWF(u)$: The space coordinate is zero and the frequency directions are exactly the ``irregular'' frequency directions of $u$. 

In the literature there are several concepts of global wave front sets apart from the Gabor wave front set. 
There is a parametrized version \cite{PWWahlberg1} and an Gelfand--Shilov version \cite{PWCarypis1} of the same idea. 
Melrose \cite{PWMelrose1} introduced the scattering wave front set which was used by Coriasco and Maniccia \cite{PWCoriasco1} for propagation of singularities, cf. also \cite{PWCoriasco2}. 
Cappiello \cite{PWCappiello1} has studied the corresponding concept in a Gelfand--Shilov framework. 

The paper is organized as follows. Section \ref{PWsec:preliminaries} contains notation and background, 
in Section \ref{PWsec:Gaborclassical} we prove the main results, and finally in 
Section \ref{PWsec:propsing} we discuss how the results from Section \ref{PWsec:Gaborclassical} can be applied to propagation of singularities for certain Schr\"odinger type evolution equations. 

\section{Preliminaries}
\label{PWsec:preliminaries}

An open ball of radius $r>0$ and center at the origin is denoted $B_r$.
The unit sphere in $\PWrr d$ is denoted $S_{d-1}$.
We write $f (x) \lesssim g (x)$ provided there exists $C>0$ such that $f (x) \leqslant C g(x)$ for all $x$ in the domain of $f$ and $g$. The Japanese bracket on $\PWrr d$ is defined by $\langle x \rangle = \sqrt{1+|x|^2}$. Peetre's inequality is
\begin{equation}\label{PWeq:peetre}
\PWeabs{x+y}^{s}\lesssim \PWeabs{x}^{s} \PWeabs{y}^{|s|}, \quad s \in \PWro. 
\end{equation}

The Fourier transform on the Schwartz space $\PWcS(\PWrr d)$ is normalized as
\begin{equation*}
\PWcF f(\xi) = \widehat f (\xi) = \int_{\PWrr d} f(x) e^{- i \langle x, \xi \rangle} d x, \quad \xi \in \PWrr d, \quad f \in \PWcS(\PWrr d),
\end{equation*}
where $\langle \cdot, \cdot \rangle$ is the inner product on $\PWrr d$, 
and extended to its dual, the tempered distributions $\PWcS'(\PWrr d)$. 
The inner product $(\cdot,\cdot)$ on $L^2 (\PWrr d) \times L^2(\PWrr d)$ is conjugate linear in the second argument. 
We use this notation also for the conjugate linear action of $\PWcS'(\PWrr d)$ on $\PWcS(\PWrr d)$.

Some notions of time-frequency analysis and pseudodifferential operators on $\PWrr d$ are recalled \cite{PWFolland1,PWGrochenig1,PWHormander0,PWLerner1,PWNicola1,PWShubin1}.
Let $u \in \PWcS'(\PWrr d)$ and $\psi \in \PWcS(\PWrr d) \setminus \{ 0 \}$. The \emph{short-time Fourier transform} (STFT) $V_\psi u $ of $u$ with respect to the window function $\psi$, is defined as
\begin{equation*}
V_\psi u :\ \PWrr {2d} \rightarrow \PWco, \quad z \mapsto V_\psi u(x,\xi) = (u, \Pi(z) \psi ),
\end{equation*}
where $\Pi(z)=M_\xi T_x$, $z=(x,\xi) \in \PWrr {2d}$ is the time-frequency shift composed of the translation operator $T_x \psi(y)=\psi(y-x)$ and the modulation operator $M_\xi \psi(y)=e^{i \langle y , \xi \rangle} \psi(y)$.

We have $V_\psi u \in C^\infty(\PWrr {2d})$ and by \cite[Theorem~11.2.3]{PWGrochenig1} there exists $m \geqslant 0$ such that 
\begin{equation}\label{PWeq:STFTbound}
|V_\psi u(z)| \lesssim \PWeabs{z}^m, \quad z \in \PWrr {2d}.
\end{equation}
The order of growth $m$ does not depend on $\psi \in \PWcS(\PWrr d) \setminus 0$. 
Let $\psi \in \PWcS(\PWrr d)$ satisfy $\| \psi \|_{L^2}=1$.
The Moyal identity
\begin{equation*}
(u,g)= (2\pi)^{-d} \int_{\PWrr {2d}} V_\psi u(z) \, \overline{V_\psi g(z)} \, d z, \quad g \in \PWcS(\PWrr d), \quad u \in \PWcS'(\PWrr d),
\end{equation*}
is sometimes written
\begin{equation}\label{PWeq:moyal}
u = (2\pi)^{-d} \int_{\PWrr {2d}} V_\psi u(x,\xi) \, M_\xi T_x \psi \, d x \, d \xi, \quad u \in \PWcS'(\PWrr d),
\end{equation}
with action understood to take place under the integral. In this form it is an inversion formula for the STFT.

Let $a \in C^\infty (\PWrr {2d})$ and $m \in \PWro$. Then $a$ is a \emph{Shubin symbol} of order $m$, denoted $a\in \Gamma^m$, if for all $\alpha,\beta \in \PWnn d$ there exists a constant $C_{\alpha\beta}>0$ such that
\begin{equation}\label{PWeq:shubinineq}
|\partial_x^\alpha \partial_\xi^\beta a(x,\xi)| \leqslant C_{\alpha\beta}\langle (x,\xi)\rangle^{m-|\alpha|-|\beta|}, \quad x,\xi \in \PWrr d.
\end{equation}
The Shubin symbols form a Fr\'echet space where the seminorms are given by the smallest possible constants in \eqref{PWeq:shubinineq}.

For $a \in \Gamma^m$ a pseudodifferential operator in the Weyl quantization is defined by
\begin{equation*}
a^w(x,D) u(x)
= (2\pi)^{-d}  \int_{\PWrr {2d}} e^{i \langle x-y, \xi \rangle} a \left(\frac{x+y}{2},\xi \right) \, u(y) \, d y \, d \xi, \quad u \in \PWcS(\PWrr d),
\end{equation*}
when $m<-d$. The definition extends to general $m \in \PWro$ if the integral is viewed as an oscillatory integral.
The operator $a^w(x,D)$ then acts continuously on $\PWcS(\PWrr d)$ and extends by duality to a continuous operator on $\PWcS'(\PWrr d)$.
By Schwartz's kernel theorem the Weyl quantization procedure may be extended to a weak formulation which yields operators $a^w(x,D):\PWcS(\PWrr{d}) \rightarrow \PWcS'(\PWrr{d})$, even if $a$ is only an element of $\PWcS'(\PWrr{2d})$.

The phase space $T^* \PWrr d$ is a symplectic vector space equipped with the 
canonical symplectic form
\begin{equation*}
\sigma((x,\xi), (x',\xi')) = \langle x' , \xi \rangle - \langle x, \xi' \rangle, \quad (x,\xi), (x',\xi') \in T^* \PWrr d.
\end{equation*}
The real symplectic group $\PWSp(d,\PWro)$ is the set of matrices in $\PWGL(2d,\PWro)$ that leaves $\sigma$ invariant. 
To each $\chi \in \PWSp(d,\PWro)$ is associated an operator $\mu(\chi)$ which is unitary on $L^2(\PWrr d)$, and determined up to a complex factor of modulus one, such that
\begin{equation*}
\mu(\chi)^{-1} a^w(x,D) \, \mu(\chi) = (a \circ \chi)^w(x,D), \quad a \in \PWcS'(\PWrr {2d})
\end{equation*}
(cf. \cite{PWFolland1,PWHormander0}).
The operators $\mu(\chi)$ are homeomorphisms on $\PWcS$ and on $\PWcS'$, 
and are called metaplectic operators.

The metaplectic representation is the mapping $\PWSp(d,\PWro) \ni \chi \mapsto \mu(\chi)$ which is a homomorphism modulo sign
\begin{equation*}
\mu(\chi_1) \mu(\chi_2) = \pm \mu(\chi_1 \chi_2), \quad \chi_1,\chi_2 \in \PWSp(d,\PWro). 
\end{equation*}
Two ways to overcome the sign ambiguity are to pass to a double-valued representation \cite{PWFolland1}, or to a representation of the so called two-fold covering group of $\PWSp(d,\PWro)$. 
The latter group is called the metaplectic group $\PWMp(d,\PWro)$. 
The two-to-one projection $\pi: \PWMp(d,\PWro) \rightarrow \PWSp(d,\PWro)$ is $\mu(\chi)\mapsto \chi$ whose kernel is $\{\pm 1\}$. The sign ambiguity may be fixed (hence it is possible to choose a section of $\pi$) along a continuous path $\PWro \ni t\mapsto \chi_t \in \PWSp(d,\PWro)$. 
This involves the so called Maslov factor \cite{PWLeray1}.

\section{The Gabor and the classical wave front sets}
\label{PWsec:Gaborclassical}

First we define the Gabor wave front set $\PWWF_G$ introduced in \cite{PWHormander1} and further elaborated in \cite{PWRodino1}. 

\begin{defn}\label{PWdef:Gaborwavefront}
Let $\varphi \in \PWcS(\PWrr d) \setminus 0$, $u \in \PWcS'(\PWrr d)$ and $z_0 \in T^* \PWrr d \setminus 0$. 
Then $z_0 \notin \PWWF_G(u)$ if there exists an open conic set $\Gamma \subseteq T^* \PWrr d \setminus 0$ such that $z_0 \in \Gamma$ and 
\begin{equation}\label{PWeq:conedecay}
\sup_{z \in \Gamma} \PWeabs{z}^N | V_\varphi u(z)| < \infty, \quad N \geqslant 0. 
\end{equation}
\end{defn}

This means that $V_\varphi u$ decays rapidly (super-polynomially) in $\Gamma$. 
The condition \eqref{PWeq:conedecay} is independent of $\varphi \in \PWcS(\PWrr d) \setminus 0$, in the sense that super-polynomial decay will hold also for $V_\psi u$ if $\psi \in \PWcS(\PWrr d) \setminus 0$, in a possibly smaller cone containing $z_0$. 
The Gabor wave front set is a closed conic subset of $T^*\PWrr d  \setminus 0$. 
By \cite[Proposition~2.2]{PWHormander1} it is symplectically invariant in the sense of 
\begin{equation}\label{PWeq:metaplecticWFG}
\PWWF_G( \mu(\chi) u) = \chi \PWWF_G(u), \quad \chi \in \PWSp(d, \PWro), \quad u \in \PWcS'(\PWrr d).
\end{equation}

The Gabor wave front set is naturally connected to the definition of the $C^\infty$ wave front set \cite[Chapter~8]{PWHormander0}, often called just the wave front set and denoted $\PWWF$. 
A point in the phase space $(x_0,\xi_0) \in T^* \PWrr d$ such that $\xi_0 \neq 0$ satisfies $(x_0,\xi_0) \notin \PWWF(u)$ 
if there exists $\varphi \in C_c^\infty(\PWrr d)$ such that $\varphi(0) \neq 0$, 
an open conical set $\Gamma_2 \subseteq \PWrr d \setminus 0$ such that $\xi_0 \in \Gamma_2$, and 
\begin{equation*}
\sup_{\xi \in \Gamma_2} \PWeabs{\xi}^N | V_\varphi u(x_0,\xi)| < \infty, \quad N \geqslant 0. 
\end{equation*}

The difference compared to $\PWWF_G(u)$ is that the $C^\infty$ wave front set $\PWWF(u)$ is defined in terms of super-polynomial decay in the frequency variable, for $x_0$ fixed, instead of super-polynomial decay in an open cone in the phase space $T^* \PWrr d$ containing the point of interest. 

Pseudodifferential operators with Shubin symbols are microlocal with respect to the Gabor wave front set. 
In fact we have by  \cite[Proposition~2.5]{PWHormander1} 
\begin{equation*}
\PWWF_G (a^w(x,D) u) \subseteq  \PWWF_G (u)
\end{equation*}
provided $a \in \Gamma^m$ and $u \in \PWcS'(\PWrr d)$. 

In the next result we relate the Gabor wave front set with the $C^\infty$ wave front set for a tempered distribution. 
We use the notation $\pi_2(x,\xi) = \xi$ for the projection $\pi_2: T^* \PWrr d \to \PWrr d$ onto the covariable. 

\begin{prop}\label{PWprop:WFGinclusion1}
If $u \in \PWcS'(\PWrr d)$ then 
\begin{equation*}
\{ 0 \} \times \pi_2 \PWWF(u) \subseteq \PWWF_G (u) . 
\end{equation*} 
\end{prop}

\begin{proof}
Suppose $|\xi_0| = 1$ and $(0,\xi_0) \notin \PWWF_G(u)$. 
Let $\varphi \in C_c^\infty(\PWrr d)$ satisfy $\varphi(0) \neq 0$.  
There exists an open conic set $\Gamma \subseteq T^* \PWrr d \setminus 0$ such that $(0,\xi_0) \in \Gamma$ and 
\begin{equation*}
\sup_{(x,\xi) \in \Gamma} \langle (x,\xi) \rangle^N  |V_\varphi u(x,\xi)| < \infty, \quad N \geqslant 0. 
\end{equation*}
We have to show that $(x_0,\xi_0) \notin \PWWF(u)$ for all $x_0 \in \PWrr d$. 

Let $x_0 \in \PWrr d$.
Define for $\PWep>0$ the open conic set containing $\xi_0$
\begin{equation*}
\Gamma_{2,\PWep} = \left\{ \xi \in \PWrr d \setminus 0: \, \left| \frac{\xi}{|\xi|} - \xi_0 \right| < \PWep \right\} \subseteq \PWrr d \setminus 0. 
\end{equation*}
We claim that there exists $\PWep>0$ such that 
\begin{equation}\label{PWconeinclusion1}
(\{ x_0 \} \times \Gamma_{2,\PWep}) \setminus B_{1/\PWep} \subseteq \Gamma
\end{equation}
holds. 
To prove this we assume for a contradiction that there exists $\xi_n \in \Gamma_{2,1/n}$ such that $|(x_0,\xi_n)| \geqslant n$ and $(x_0,\xi_n) \notin \Gamma$ for all $n \in \PWno$. 
Since $\Gamma$ is conic we have $|\xi_n|^{-1} (x_0,\xi_n) \notin \Gamma$. 
The sequence  $\left( |\xi_n|^{-1} (x_0,\xi_n) \right)_{n} \subseteq \PWrr {2d}$ is bounded so we get for a subsequence 
\begin{equation*}
\left( \frac{x_0}{|\xi_{n_j}|},\frac{\xi_{n_j}}{|\xi_{n_j}|} \right) \rightarrow (x,\xi), \quad j \rightarrow \infty, 
\end{equation*}
where $(x,\xi) \notin \Gamma$ due to $\Gamma$ being open, and $|\xi|=1$.  
From $|\xi_n| \geqslant n - |x_0|$ we may conclude that $x=0$. 
From $\xi_n \in \Gamma_{2,1/n}$ for all $n \in \PWno$ we obtain $\xi=\xi_0$, which gives $(0,\xi_0) \notin \Gamma$. 
This is a contradiction which shows that \eqref{PWconeinclusion1} must hold for some $\PWep>0$. 

Thus we have for $N \geqslant 0$ arbitrary
\begin{align*}
\sup_{\xi \in \Gamma_{2,\PWep}, \ |\xi| \geqslant \PWep^{-1} + |x_0|} \langle \xi \rangle^N  |V_\varphi u(x_0,\xi)| 
& \leqslant \sup_{\xi \in \Gamma_{2,\PWep}, \ |\xi| \geqslant \PWep^{-1} + |x_0|} \langle (x_0,\xi) \rangle^N  |V_\varphi u(x_0,\xi)| \\
& \leqslant \sup_{(x,\xi) \in \Gamma} \langle (x,\xi) \rangle^N  |V_\varphi u(x,\xi)| 
< \infty
\end{align*} 
which shows that $(x_0,\xi_0) \notin \PWWF(u)$. 
\end{proof}

\begin{prop}\label{PWprop:WFGinclusion2}
If $u \in \PWcE'(\PWrr d) + \PWcS(\PWrr d)$ then 
\begin{equation}\label{PWeq:WFGinclusion1}
\PWWF_G (u) \subseteq \{ 0 \} \times \pi_2 \PWWF(u). 
\end{equation} 
\end{prop}

\begin{proof}
We may assume $u \in \PWcE'(\PWrr d)$. 
We start with the less precise inclusion 
\begin{equation}\label{PWeq:subinclusion2}
\PWWF_G (u) \subseteq \{ 0 \} \times \PWrr d \setminus 0. 
\end{equation} 
Suppose $0 \neq (x_0,\xi_0) \notin \{ 0 \} \times \PWrr d \setminus 0$. Then $x_0 \neq 0$ so for some $C>0$ we have $(x_0,\xi_0) \in \Gamma = \{ (x,\xi) \in T^* \PWrr d \setminus 0: \, |\xi| < C |x| \} \subseteq T^* \PWrr d \setminus 0$ which is an open conic set. 
If we pick $\varphi \in C_c^\infty(\PWrr d)$ we have $V_\varphi u(x,\xi) = 0$ if $|x| \geqslant A$ for $A>0$ sufficiently large due to $u \in \PWcE'(\PWrr d)$. 

Since we have the bound 
\eqref{PWeq:STFTbound} for some $m \geqslant 0$, 
we obtain for any $N \geqslant 0$
\begin{equation}\label{eq:STFTestimate1}
\begin{aligned}
\sup_{(x,\xi) \in \Gamma} \PWeabs{(x,\xi)}^N  |V_\varphi u(x,\xi)| 
& = \sup_{(x,\xi) \in \Gamma, \ |x| \leqslant A} \PWeabs{(x,\xi)}^N  |V_\varphi u (x,\xi)| \\
& \lesssim \sup_{(x,\xi) \in \Gamma, \ |x| \leqslant A} \PWeabs{x}^{N+m} < \infty. 
\end{aligned} 
\end{equation} 
It follows that $(x_0,\xi_0) \notin \PWWF_G(u)$, which proves the inclusion \eqref{PWeq:subinclusion2}. 

In order to show the sharper inclusion \eqref{PWeq:WFGinclusion1}, suppose $0 \neq (x_0,\xi_0) \notin \{ 0 \} \times \pi_2 \PWWF(u)$. 
Then either $x_0 \neq 0$ or $\xi_0 \notin \pi_2 \PWWF(u)$. If $x_0 \neq 0$ then by \eqref{PWeq:subinclusion2} we have $(x_0,\xi_0) \notin \PWWF_G(u)$. 
Therefore we may assume that $x_0=0$ and $\xi_0 \notin \pi_2 \PWWF(u)$, and our goal is to show $(0,\xi_0) \notin \PWWF_G(u)$, which will prove \eqref{PWeq:WFGinclusion1}. 

By \cite[Proposition~8.1.3]{PWHormander0} we have $\pi_2 \PWWF(u) = \Sigma(u)$, where $\Sigma(u) \subseteq \PWrr d \setminus 0$ is a closed conic set defined as follows.
A point $\eta \in \PWrr d \setminus 0$ satisfies $\eta \notin \Sigma(u)$ if $\eta \in \Gamma_2$ where $\Gamma_2 \subseteq \PWrr d \setminus 0$ is open and conic, and 
\begin{equation}\label{PWeq:frequencydecay1}
\sup_{\xi \in \Gamma_2} \PWeabs{\xi}^N |\widehat u(\xi)| < \infty, \quad N \geqslant 0. 
\end{equation} 

Thus we have $\xi_0 \notin \Sigma(u)$, so there exists an open conic set $\Gamma_2 \subseteq \PWrr d \setminus 0$ such that $\xi_0 \in \Gamma_2$, and \eqref{PWeq:frequencydecay1} holds. 
Let $\Gamma_2' \subseteq \PWrr d \setminus 0$ be an open conic set such that $\xi_0 \in \Gamma_2'$ and $\overline{\Gamma_2' \cap S_{d-1}} \subseteq \Gamma_2$. 

We have 
\begin{equation*}
V_\varphi u (x,\xi) = \widehat{u T_x \overline{\varphi}} (\xi)
= (2 \pi)^{-d} \widehat{u} * \widehat{T_x \overline{\varphi}} (\xi)
\end{equation*} 
which gives
\begin{equation}\label{PWeq:STFTconvolution1}
|V_\varphi u (x,\xi)| 
\lesssim |\widehat{u}| * |g| (\xi), \quad x, \ \xi \in \PWrr d, 
\end{equation} 
where $g (\xi)= \widehat \varphi(-\xi) \in \PWcS(\PWrr d)$.  
By the Paley--Wiener--Schwartz theorem we have for some $m \geqslant 0$
\begin{equation}\label{PWeq:PWS}
|\widehat{u} (\xi)| \lesssim \PWeabs{\xi}^m, \quad \xi \in \PWrr d. 
\end{equation} 
Define the open conic set
\begin{equation*}
\Gamma = \{ (x,\xi) \in T^* \PWrr d \setminus 0: \, |x| < |\xi|, \, \xi \in \Gamma_2' \}. 
\end{equation*} 
Then $(0,\xi_0) \in \Gamma$. 
To prove $(0,\xi_0) \notin \PWWF_G(u)$ it thus suffices to show 
\begin{equation*}
\sup_{(x,\xi) \in \Gamma} \PWeabs{(x,\xi)}^N |V_\varphi u (x,\xi)| <  \infty, \quad N \geqslant 0.   
\end{equation*} 

In turn, these estimates will by \eqref{PWeq:STFTconvolution1} follow from the estimates
\begin{equation}\label{PWeq:FTconvolution1}
\sup_{\xi \in \Gamma_2'} \PWeabs{\xi}^N (|\widehat{u}| * |g|) (\xi) <  \infty, \quad N \geqslant 0.   
\end{equation} 
Let $\PWep>0$ and split the convolution integral as 
\begin{equation*}
|\widehat{u}| * |g| (\xi) = \int_{\PWeabs{\eta} \leqslant \PWep \PWeabs{\xi} } |\widehat{u} (\xi-\eta)| \, |g(\eta)| \, d \eta
+ \int_{\PWeabs{\eta} > \PWep \PWeabs{\xi} } |\widehat{u} (\xi-\eta)| \, |g(\eta)| \, d \eta. 
\end{equation*} 

If $\xi \in \Gamma_2'$, $|\xi| \geqslant 1$ and $\PWeabs{\eta} \leqslant \PWep \PWeabs{\xi}$ for $\PWep>0$ sufficiently small, then $\xi-\eta \in \Gamma_2$. 
Using \eqref{PWeq:peetre} and \eqref{PWeq:frequencydecay1} we thus obtain if $\xi \in \Gamma_2'$ and $|\xi| \geqslant 1$ for any $N \geqslant 0$
\begin{equation}\label{PWeq:integralestimate1}
\begin{aligned}
\int_{\PWeabs{\eta} \leqslant \PWep \PWeabs{\xi} } |\widehat{u} (\xi-\eta)| \, |g(\eta)| \, d \eta
& \lesssim \int_{\PWeabs{\eta} \leqslant \PWep \PWeabs{\xi} } \PWeabs{\xi-\eta}^{-N} \, |g(\eta)| \, d \eta \\
& \lesssim \PWeabs{\xi}^{-N} \int_{\PWrr d} \PWeabs{\eta}^{N} \, |g(\eta)| \, d \eta \\
& \lesssim \PWeabs{\xi}^{-N} 
\end{aligned}
\end{equation}
since $g \in \PWcS(\PWrr d)$. 

The remaining integral we estimate using \eqref{PWeq:PWS}. 
We thus have for any $\xi \in \PWrr d$ and any $N \geqslant 0$
\begin{equation}\label{PWeq:integralestimate2}
\begin{aligned}
\int_{\PWeabs{\eta} > \PWep \PWeabs{\xi} } |\widehat{u} (\xi-\eta)| \, |g(\eta)| \, d \eta
& \lesssim \int_{\PWeabs{\eta} > \PWep \PWeabs{\xi} } \PWeabs{\xi-\eta}^{m} \, |g(\eta)| \, d \eta \\
& \lesssim \PWeabs{\xi}^{-N} \int_{\PWeabs{\eta} > \PWep \PWeabs{\xi} } \PWeabs{\xi}^{m+N} \PWeabs{\eta}^{m} \, |g(\eta)| \, d \eta \\
& \lesssim \PWeabs{\xi}^{-N} \int_{\PWeabs{\eta} > \PWep \PWeabs{\xi} } \PWeabs{\eta}^{2 m+N} \, |g(\eta)| \, d \eta \\
& \lesssim \PWeabs{\xi}^{-N}. 
\end{aligned}
\end{equation}
Combining \eqref{PWeq:integralestimate1} and \eqref{PWeq:integralestimate2} shows that the estimates \eqref{PWeq:FTconvolution1} hold, which as earlier explained proves $(0,\xi_0) \notin \PWWF_G(u)$. 
This proves the inclusion \eqref{PWeq:WFGinclusion1}. 
\end{proof}

\begin{cor}\label{PWcor:WFGcompact}
If $u \in \PWcE'(\PWrr d) + \PWcS(\PWrr d)$ then 
\begin{equation*}
\PWWF_G (u) = \{ 0 \} \times \pi_2 \PWWF(u) = \{ 0 \} \times \Sigma(u). 
\end{equation*} 
\end{cor}

The next result is a consequence of \cite[Proposition~2.7]{PWHormander1}. 
We include a proof here in order to give a self-contained account, and also in order to show an alternative proof technique. 

\begin{prop}\label{PWprop:smooth}
If $u \in \PWcS'(\PWrr d)$ and $\PWWF_G(u) \cap ( \{ 0 \} \times \PWrr d) = \emptyset$ then $u \in C^\infty(\PWrr d)$ and there exists $L \geqslant 0$ such that 
\begin{equation}\label{PWeq:polynomialbound1}
|\partial^\alpha u(x)| \leqslant C_\alpha \PWeabs{x}^{L + |\alpha|}, \quad x \in \PWrr d, \quad \alpha \in \PWnn d, \quad C_\alpha>0. 
\end{equation}
\end{prop}

\begin{proof}
From the assumptions it follows that for some $C>0$ we have 
\begin{equation*}
\PWWF_G(u) \subseteq \Gamma := \{ (x,\xi) \in T^* \PWrr d \setminus 0: \, |\xi| < C |x| \}. 
\end{equation*}
Let $\psi \in \PWcS(\PWrr d)$ satisfy $\| \psi \|_{L^2} = 1$. 
We use the Moyal identity \eqref{PWeq:moyal}
and show that the integral for $\partial^\alpha u$ is absolutely convergent for any $\alpha \in \PWnn d$. 
Thus we write formally
\begin{equation}\label{PWeq:STFTreconstruction}
\partial^\alpha u (y) = (2\pi)^{-d} \sum_{\beta \leqslant \alpha} \binom{\alpha}{\beta} \int_{\PWrr {2d}} V_\psi u(x,\xi) \, (i\xi)^\beta e^{i \langle \xi,y \rangle} \partial^{\alpha-\beta} \psi(y-x) \, d x \, d \xi. 
\end{equation}

We split the integral in two parts. Since $\PWrr {2d} \setminus \Gamma \subseteq T^* \PWrr d$ is a closed conic set that does not intersect $\PWWF_G(u)$ 
we have for any $N \geqslant 0$ and any $k \geqslant 0$
\begin{equation}\label{PWeq:integralpiece1}
\begin{aligned}
& \left| \int_{\PWrr {2d} \setminus \Gamma} V_\psi u(x,\xi) \, (i\xi)^\beta e^{i \langle \xi,y \rangle} \partial^{\alpha-\beta} \psi(y-x) \, d x \, d \xi \right| \\
& \leqslant \int_{\PWrr {2d} \setminus \Gamma} |V_\psi u(x,\xi)| \, \PWeabs{\xi}^{|\alpha|} \, |\partial^{\alpha-\beta} \psi(y-x)| \, d x \, d \xi \\
& \lesssim \int_{\PWrr {2d} \setminus \Gamma} \PWeabs{(x,\xi)}^{-N} \, \PWeabs{\xi}^{|\alpha|} \, |\partial^{\alpha-\beta} \psi(y-x)| \, d x \, d \xi \\
& \lesssim \int_{\PWrr {2d} \setminus \Gamma} \PWeabs{\xi}^{|\alpha|-N} \, \PWeabs{x}^{N} \, \PWeabs{y-x}^{-k} \, d x \, d \xi \\
& \lesssim \PWeabs{y}^{k} \int_{\PWrr {2d}} \PWeabs{\xi}^{|\alpha|-N} \, \PWeabs{x}^{N-k} \, d x \, d \xi \\
& \lesssim \PWeabs{y}^{k} 
\end{aligned}
\end{equation}
provided $N > d + |\alpha|$ and $k > d + N. $

For the remaining part of the integral we use the estimate \eqref{PWeq:STFTbound}
for some $m \geqslant 0$. 
Using $|\xi| < C |x|$ when $(x,\xi) \in \Gamma$, this gives for any $k \geqslant 0$
\begin{equation}\label{PWeq:integralpiece2}
\begin{aligned}
& \left| \int_{\Gamma} V_\psi u(x,\xi) \, (i\xi)^\beta e^{i \langle \xi,y \rangle} \partial^{\alpha-\beta} \psi(y-x) \, d x \, d \xi \right| \\
& \lesssim \int_{|\xi| < C |x|} \PWeabs{(x,\xi)}^m \, \PWeabs{\xi}^{|\alpha|} \, |\partial^{\alpha-\beta} \psi(y-x)| \, d x \, d \xi \\
& \lesssim \int_{|\xi| < C |x|} \PWeabs{\xi}^{-d-1} \, \PWeabs{x}^{|\alpha|+m+d+1} \, |\partial^{\alpha-\beta} \psi(y-x)| \, d x \, d \xi \\
& \lesssim \int_{\PWrr {2d}} \PWeabs{\xi}^{-d-1} \, \PWeabs{x}^{|\alpha|+m+d+1} \, \PWeabs{y-x}^{-k} \, d x \, d \xi \\
& \lesssim \PWeabs{y}^{k} \int_{\PWrr {2d}} \PWeabs{\xi}^{-d-1} \, \PWeabs{x}^{|\alpha|+m+d+1-k} \, d x \, d \xi \\
& \lesssim \PWeabs{y}^{k} 
\end{aligned}
\end{equation}
provided $k > |\alpha|+m+2d+1$. 
Combining \eqref{PWeq:integralpiece1} and \eqref{PWeq:integralpiece2} shows in view of \eqref{PWeq:STFTreconstruction} that $u \in C^\infty(\PWrr d)$ and the estimate \eqref{PWeq:polynomialbound1} follows. 
\end{proof}

\section{Propagation of singularities for Schr\"odinger equations}
\label{PWsec:propsing}

In this section we discuss briefly the initial value Cauchy problem for a Schr\"odinger equation of the form
\begin{equation}\label{PWeq:schrodeq}
\left\{
\begin{array}{rl}
\partial_t u(t,x) + q^w(x,D) u (t,x) & = 0, \\
u(0,\cdot) & = u_0, 
\end{array}
\right.
\end{equation}
where $u_0 \in \PWcS'(\PWrr d)$, $t \geqslant 0$ and $x \in \PWrr d$.
The Weyl symbol of the Hamiltonian $q^w(x,D)$ is a quadratic form
\begin{equation*}
q(x,\xi) = \langle (x, \xi), Q (x, \xi) \rangle, \quad x, \, \xi \in \PWrr d, 
\end{equation*}
where $Q \in \PWcc {2d \times 2d}$ is a symmetric matrix with $\PWre Q \geqslant 0$. 

The \emph{Hamilton map} $F$ corresponding to $q$ is defined by 
$$
F = \PWJ Q \in \PWcc {2d \times 2d}
$$
where 
$$
\PWJ = 
\left( 
\begin{array}{rr}
0 \ & I \\
-I \ & 0
\end{array}
\right) \in \PWrr {2d \times 2d}
$$
is a symplectic matrix that is a cornerstone in symplectic linear algebra. 

The equation \eqref{PWeq:schrodeq} is solved for $t \geqslant 0$ by 
$$
u(t,x) = e^{-t q^w(x,D)} u_0(x)
$$
where the propagator $e^{-t q^w(x,D)}$ is defined in terms of semigroup theory \cite{PWHormander2,PWYosida1}. 

According to \cite[Theorem~6.2]{PWRodino2} the Gabor wave front set propagates as stated in the following result. 
Let $q$ be a quadratic form on $T^* \PWrr d$ defined by a symmetric matrix $Q \in \PWcc {2d \times 2d}$, $\PWre Q \geqslant 0$ and $F = \PWJ Q$. 
Then for $u \in \PWcS'(\PWrr d)$ and $t>0$
\begin{equation*}
\PWWF_G ( e^{-t q^w(x,D)} u ) \subseteq \left( \left( e^{2 t \PWim F} ( \PWWF_G (u) \cap S ) \right) \cap S \right) \setminus 0
\end{equation*}
where the \emph{singular space} is defined by 
\begin{equation}\label{singspace}
S=\Big(\bigcap_{j=0}^{2d-1} \PWKer\big[\PWre F(\PWim F)^j \big]\Big) \cap T^* \PWrr d \subseteq T^* \PWrr d. 
\end{equation}
Under the additional assumption on the Poisson bracket $\{q,\overline q\} = 0$, \cite[Corollary~6.3]{PWRodino2} says that 
$S = \PWKer (\PWre F )$ and hence
\begin{equation*}
\PWWF_G ( e^{-t q^w(x,D)} u ) \subseteq \left( \left( e^{2 t \PWim F} ( \PWWF_G (u) \cap \PWKer (\PWre F ) ) \right) \cap  \PWKer (\PWre F ) \right) \setminus 0
\end{equation*}
for $u \in \PWcS'(\PWrr d)$ and $t>0$. 

If we combine these results with Proposition \ref{PWprop:smooth} we get the following consequence.

\begin{cor}
Let $q$ be a quadratic form on $T^* \PWrr d$ defined by a symmetric matrix $Q \in \PWcc {2d \times 2d}$, $\PWre Q \geqslant 0$ and $F = \PWJ Q$. 
If  
\begin{equation*}
S \cap (\{0\} \times \PWrr d)= \{ 0 \}, 
\end{equation*}
which if $\{q,\overline q\} = 0$ reads 
\begin{equation*}
\PWKer (\PWre F ) \cap (\{0\} \times \PWrr d)= \{ 0 \}, 
\end{equation*}
then for $u \in \PWcS'(\PWrr d)$ and $t>0$ we have $e^{-t q^w(x,D)} u \in C^\infty(\PWrr d)$, and
\begin{equation*}
|\partial^\alpha e^{-t q^w(x,D)} u (x)| \lesssim \PWeabs{x}^{L_t + |\alpha|}, \quad x \in \PWrr d, \quad \alpha \in \PWnn d,  
\end{equation*}
for some $L_t \geqslant 0$.  
\end{cor}

Next we specialize to the Cauchy initial value problem for the harmonic oscillator Schr\"odinger equation 
\begin{equation}\label{PWeq:harmonicoscillator}
\left\{
\begin{array}{rl}
\partial_t u(t,x) + i(|x|^2- \Delta_x) u(t,x) & = 0, \qquad t \geqslant 0, \quad x \in \PWrr d, \\
u(0,\cdot) & = u_0 \in \PWcS'(\PWrr d).  
\end{array}
\right.
\end{equation}
This problem is a particular case of the general problem \eqref{PWeq:schrodeq} with $Q = i I_{2d}$. 
When $\PWre Q = 0$ the propagator is the unitary group $e^{- t q^w(x,D)} = \mu(e^{2 t \PWim F})$, $t \in \PWro$, on $L^2(\PWrr d)$ \cite{PWRodino2},  and
\begin{equation*}
\PWWF_G ( e^{- t q^w(x,D)} u_0 ) = e^{2 t \PWim F} \PWWF_G(u_0), \quad t \in \PWro, \quad u_0 \in \PWcS'(\PWrr d). 
\end{equation*}
The propagation is exact and time is reversible. 
This result is a consequence of the metaplectic representation and \eqref{PWeq:metaplecticWFG}. 
The quoted results \cite[Theorem~6.2 and Corollary~6.3]{PWRodino2} are not needed. 

For the equation \eqref{PWeq:harmonicoscillator} we thus have periodic propagation of the Gabor wave front set:
\begin{equation}\label{PWeq:propagationharmosc}
\PWWF_G(e^{- t q^w(x,D)} u_0) 
= e^{2 t \PWJ} \PWWF_G(u_0)
= \left(
\begin{array}{cc}
\cos(2t) I_d & \sin(2t) I_d \\
-\sin(2t) I_d & \cos(2t) I_d \\
\end{array}
\right) \PWWF_G(u_0), \quad t \in \PWro
\end{equation}
(cf. \cite[Example~7.5]{PWRodino2}). 

The following result gives a partial explanation, from the point of view of the Gabor wave front set, of Weinstein's
\cite{PWWeinstein1} and Zelditch's \cite{PWZelditch1} results on the propagation of the $C^\infty$ wave front set for the harmonic oscillator. 
The result says that a compactly supported initial datum will give a solution that is smooth except for a lattice on the time axis. 
At the points of the lattice the propagator is the identity or coordinate reflection $f (x) \mapsto f(- x)$,
which gives precise propagation of singularities by possible sign changes. 
Note that we do not allow as general potentials as in \cite{PWWeinstein1,PWZelditch1}. 

\begin{prop}
Consider the equation \eqref{PWeq:harmonicoscillator} and suppose $u_0 \in \PWcE'(\PWrr d) + \PWcS(\PWrr d)$. 
If $t \notin (\pi/2) \PWzo$ then $e^{- t q^w(x,D)} u_0 \in C^\infty(\PWrr d)$ and for some $L_t \geqslant 0$
\begin{equation}\label{PWeq:growthestimate1}
|\partial^\alpha e^{-t q^w(x,D)} u_0 (x)| \lesssim \PWeabs{x}^{L_t + |\alpha|}, \quad x \in \PWrr d, \quad \alpha \in \PWnn d.   
\end{equation}
If $t \in (\pi/2) \PWzo$ then 
\begin{align}
\PWWF_G( e^{- t q^w(x,D)} u_0 )  & = (-1)^{2t/\pi} \PWWF_G( u_0 ), \label{PWpropWFG1} \\
\PWWF( e^{- t q^w(x,D)} u_0 )  & = (-1)^{2t/\pi} \PWWF( u_0 ). \label{PWpropWFG2}
\end{align}
\end{prop}

\begin{proof}
Corollary \ref{PWcor:WFGcompact} implies $\PWWF_G (u_0) \subseteq \{ 0 \} \times (\PWrr d \setminus 0)$.
Combined with \eqref{PWeq:propagationharmosc} this means that 
\begin{equation*}
\PWWF_G( e^{- t q^w(x,D)} u_0 ) \cap (\{0\} \times \PWrr d) = \emptyset
\end{equation*}
unless $t \in (\pi/2) \PWzo$. 
By Proposition \ref{PWprop:smooth} we then have  $e^{- t q^w(x,D)} u_0 \in C^\infty(\PWrr d)$ and the estimates \eqref{PWeq:growthestimate1} unless $t \in (\pi/2) \PWzo$. 

If $t=\pi k/2$ for $k \in \PWzo$ then $\cos(2t) = (-1)^k$ and $\sin(2t) = 0$. 
Thus \eqref{PWeq:propagationharmosc} yields the following conclusion. 
If $t = \pi k$ for $k \in \PWzo$ then $\PWWF_G( e^{- t q^w(x,D)} u_0 ) = \PWWF_G( u_0 )$ 
whereas if $t = \pi (k+1/2)$ for some $k \in \PWzo$ then $\PWWF_G( e^{- t q^w(x,D)} u_0 ) = - \PWWF_G( u_0 )$. 
This proves \eqref{PWpropWFG1}. 

When $k = 2 t /\pi \in \PWzo$ we have 
\begin{equation*}
e^{2 t \PWJ}
= \left(
\begin{array}{cc}
\cos(2t) I_d & \sin(2t) I_d \\
-\sin(2t) I_d & \cos(2t) I_d \\
\end{array}
\right) = (-1)^k I_{2d}
\end{equation*}
and therefore the corresponding metaplectic operator is 
$\mu(e^{2 t \PWJ}) f(x) = f( (-1)^k x)$.
Thus the propagator $e^{- t q^w(x,D)} = \mu(e^{2 t \PWJ})$ is the reflection operator  $f(x) \mapsto f((-1)^k x)$ when $k = 2 t /\pi \in \PWzo$, 
which justifies $\PWWF( e^{- t q^w(x,D)} u_0 )  = (-1)^{2t/\pi} \PWWF( u_0 )$ when $2 t /\pi \in \PWzo$, that is, \eqref{PWpropWFG2}.
\end{proof}

If $t \in \pi(1 + 2 \PWzo)/4$ then $e^{2 t \PWJ} = \pm \PWJ$ and consequently 
\begin{equation*}
e^{- t q^w(x,D)} = \mu(e^{2 t \PWJ}) = \mu(\pm \PWJ) 
= 
\left\{
\begin{array}{l}
(2 \pi)^{-d/2} \PWcF \\
(2 \pi)^{d/2} \PWcF^{-1}
\end{array}
\right.
\end{equation*}
see e.g. \cite{PWCarypis1}. 
When $t \in \pi(1 + 2 \PWzo)/4$ the estimates \eqref{PWeq:growthestimate1} are thus a consequence of the Paley--Wiener--Schwartz theorem \cite[Theorem~7.3.1]{PWHormander0}.
When $t \notin \pi \PWzo/4$ the estimates \eqref{PWeq:growthestimate1} reveals that the solution satisfies similar estimates as $\PWcF \PWcE' (\PWrr d)$.

\end{document}